\documentclass{article}
\usepackage{amssymb}
\usepackage{amsfonts}
\usepackage{amsmath}

\newtheorem{theorem}{Theorem}[section]

\newtheorem{conjecture}[theorem]{Conjecture}

\newtheorem{definition}[theorem]{Definition}

\newtheorem{lemma}[theorem]{Lemma}

\newtheorem{remark}[theorem]{Remark}

\newenvironment{proof}[1][Proof]{\noindent\textbf{#1.} }{\ \rule{0.5em}{0.5em}}
\input{tcilatex}
\begin{document}

\title{ Piecewise linear actions and Zimmer's program}
\author{Shengkui Ye}
\maketitle

\begin{abstract}
We consider Zimmer's program of lattice actions on surfaces by PL
homomorphisms. It is proved that when the surface is not the torus or Klein
bottle the action of any finite-index subgroup of $\mathrm{SL}(n,\mathbb{Z}%
), $ $n\geq 5$ (more generally for any $2$-big lattice), factors through a
finite group action. The proof is based on an establishment of a PL version
of Reeb-Thurston's stability.
\end{abstract}

\section{Introduction}

This paper introduces a new approach to study lattice actions on compact
manifolds.

The special linear group $\mathrm{SL}_{n}(\mathbb{Z})$ acts on the spheres $%
S^{n-1}$ by $x\longmapsto Ax/\Vert x\Vert $ for $x\in S^{n-1}$ and $A\in 
\mathrm{SL}_{n}(\mathbb{Z})$, i.e. via the action on the space of infinite
rays based at $0\in \mathbb{R}^{n}$. It is believed that this action is
minimal in the following sense.

\begin{conjecture}
\label{conj}Any action of a finite-index subgroup of $\mathrm{SL}_{n}(%
\mathbb{Z})$ $(n\geq 3)$ on a compact $r$-manifold by homeomorphisms factors
through a finite group action if $r<n-1.$
\end{conjecture}

The smooth version of such conjecture was formulated by Farb and Shalen \cite%
{fs}, which is related to Zimmer's program \cite{z3}. Conjecture \ref{conj}
could be a special case of a general conjecture in Zimmer's program, in
which the special linear group is replaced by an arbitrary irreducible
lattice $\Gamma $ in a semisimple Lie group $G$ of $\mathbb{R}$-rank at
least 2, and the integer $n$ is replaced by a suitable integer $h(G).$ Some
form of Conjecture \ref{conj} has been discussed by Weinberger \cite{sw}.
These conjectures are part of a program to generalize the Margulis
Superrigidity Theorem to a nonlinear context.

Even for the smooth case, it is difficult to prove this conjecture. In
general, we have to assume either that the group action preserves additional
geometric structures or that the lattices themselves are special. The
following is an incomplete list of some results in this direction. For more
details of Zimmer's program and related topics, see survey articles of
Zimmer and Morris \cite{zm}, Fisler \cite{fi} and Labourie \cite{la}.

When $r=1$ and $M=S^{1}$, Witte \cite{Wi} proves that Conjecture \ref{conj}
is true for an arithmetic lattice $\Gamma $ with $\mathbb{Q}$-rank$(\Gamma
)\geq 2$.

When $r=2,$ the group action is smooth real-analytic and $M$ is a compact
surface other than the torus or Klein bottle, Farb and Shalen \cite{fs}
prove that Conjecture \ref{conj} is true for $n\geq 5$ (more generally for $%
2 $-big lattices). When the group action is smooth real-analytic and
volume-preserving, they also show that this result could be extended to all
compact surfaces.

Polterovich (see Corollary 1.1.D of \cite{Po}) proves that if $n\geq 3$,
then any action by $\mathrm{SL}(n,\mathbb{Z})$ on a closed surface other
than the sphere $S^{2}$ and the torus $T^{2}$ by area preserving
diffeomorphisms factors through a finite group action. When $r=2$ and the
group action is by area preserving diffeomorphisms, Franks and Handel \cite%
{fh} prove that Conjecture \ref{conj} is true for an almost simple group
containing a subgroup isomorphic to the three-dimensional integer Heisenberg
group (eg. any finite-index subgroup of $\mathrm{SL}_{n}(\mathbb{Z})$ for $%
n\geq 3$).

Weinberger \cite{we} shows that Conjecture \ref{conj} is true for $\mathrm{SL%
}_{n}(\mathbb{Z})$ itself and $M=T^{r},$ the torus of dimension $r.$ Bridson
and Vogtmann \cite{bv} prove Conjecture \ref{conj} for $\mathrm{SL}_{n}(%
\mathbb{Z})$ and $M=S^{r},$ the sphere of dimension $r.$

In this article, we consider lattice actions on manifolds by PL
homeomorphisms. Let's recall from Farb and Shalen \cite{fs} that an (uniform
or non-uniform) irreducible lattice $\Gamma $ in a semisimple group $G$ of $%
\mathbb{R}$-rank at least $2$ is $1$-big if either $\mathbb{Q}$-rank$(\Gamma
)\geq 2$ or the centralizer of some infinite-order element of $\Gamma $ has
a subgroup isomorphic to an irreducible lattice in some semisimple Lie group
of $\mathbb{R}$-rank at least $2.$ A lattice is $(k+1)$-big if the
centralizer of some infinite-order element has a subgroup isomorphic to a $k$%
-big lattice. For example, any finite-index subgroup of $\mathrm{SL}_{n}(%
\mathbb{Z})$ is $k$-big for $n\geq 2k+1.$ Our main result is the following.

\begin{theorem}
\label{main}Let $\Gamma $ be a $2$-big lattice and $M$ a compact surface
(with or without boundary) other than the torus or Klein bottle. Then any PL
action of $\Gamma $ on $M$ factors through a finite group action.
\end{theorem}

The proof of Theorem \ref{main} consists of two steps, similar to those of
some other results of this kind. The first step is to obtain a global fixed
point. The second step is to show that the action is trivial for a subgroup
of finite index in $\Gamma $. In the smooth real-analytic case, there are
also two such steps essentially (cf. Farb and Shalen \cite{fs}). For the
first step, it is achieved by studying the analytic set of the fixed point
set of a map and for the second step it is easy. For smooth action, the
second step is normally achieved by Reeb-Thurston's stability. Based on such
stability, Zimmer \cite{zi5} shows that the general version of Conjecture %
\ref{conj} is true if $\Gamma $ has a global fixed point. The work of Farb
and Shalen \cite{fs} is an important motivation for this work, and the
strategies of proofs in \cite{fs} often play a role in several of our
arguments in this short article.

Compared with real-analytic actions, the first step is easier for PL
actions, since fixed point sets are simplicial complexes. However, the
second step is much more difficult, since the Reeb-Thurston stability is not
available for PL actions as it is only valid for $C^{1}$ actions. Therefore,
we have to establish a PL version of Reeb-Thurston's stability. For this, we
introduce a notion of tangent spheres for PL manifolds (for details, see
Definition \ref{ts}), which plays the same role as tangent spaces for smooth
manifolds.

\begin{theorem}
\label{rt}Let $G$ be a finitely generated group acting on a connected
manifold $M$ by PL homeomorphisms with a global fixed point $p$. Suppose that

\begin{enumerate}
\item[(i)] $G$ acts trivially on the tangent sphere $R_{p};$

\item[(ii)] the first homology group $H_{1}(G;\mathbb{R})=0,$
\end{enumerate}

then the group action is trivial, i.e. $G$ fixes every point of $M.$
\end{theorem}

For group actions on $3$-manifolds, we obtain the following.

\begin{theorem}
\label{three}Let $\Gamma $ be a $2$-big lattice and $M$ be a $3$-manifold
whose boundary has nonzero Euler characteristic. Then any PL action of $%
\Gamma $ on $M$ factors through a finite group action.
\end{theorem}

\section{$k$-big lattices}

Recall the definition of $k$-big lattices from the previous section. The
following are some typical examples of $k$-big lattices (cf. \cite{fs},
Section 2):

\begin{itemize}
\item Any finite-index subgroup of $\mathrm{SL}_{n}(\mathbb{Z})$ is $k$-big
if $n\geq 2k+1.$

\item Let $\Gamma $ be a lattice in a semisimple Lie group whose root system
is not $D_{4}.$ $\Gamma $ is $k$-big if $\mathbb{Q}$-rank of $\Gamma $ is at
least $2k.$

\item Let $K$ be a finite real extension of $\mathbb{Q}$ and $\Phi $ be a
nondegenerate quadratic form of type $(p,q)$ over $K.$ If $\Phi $ is
diagonal and $(p,q)\geq (3k+3,2k+2),$ the group of $K$-integral unimodular
matrices preserving $\Phi $ is $k$-big.
\end{itemize}

A group is \emph{almost simple} if every normal subgroup is either finite
and central, or is of finite index. The Margulis Finiteness Theorem (cf. 
\cite{zi1}, Theorem 8.1) implies that a $k$-big lattice $\Gamma $ is almost
simple. Therefore, the normal subgroup of any noncentral element in $\Gamma $
is of finite index. We will use such fact several times in later arguments.

\section{A PL version of Reeb-Thurston's stability}

The following Reeb-Thurston's stability (cf. \cite{Th}) says that for a
smooth group action if the group $G$ acts trivially on the tangent space of
a global fixed point and there is no nontrivial homomorphism from $G$ to the
real number $\mathbb{R}$, then the group action is trivial.

\begin{lemma}[Reeb-Thurston Stability Theorem]
Suppose that the group $G$ is finitely generated and acts by $C^{1}$
diffeomorphisms on a connected manifold $M$, with a fixed point $p$. If

\begin{enumerate}
\item[(i)] $G$ acts trivially on the tangent space $T_{p}(M)$, and

\item[(ii)] $H_{1}(G;\mathbb{R)}=0$,

then the action is trivial (i.e., every point of $M$ is fixed by every
element of $G$).
\end{enumerate}
\end{lemma}

This is a useful lemma to tell when a group action is (globally) trivial in
terms of local information of the action.

In this section, we will prove a PL version of the Reeb-Thurston stability.
Let's recall some basic facts on PL homeomorphisms. Let $M^{n}$ be a
manifold (with or without boundary) embedded into some Euclidean space. A
homeomorphism $f$ of $M$ is piecewise linear (or PL) if there is some
triangulation of $M$ into finitely many simplices such that $f$ is affine
linear when restricted to each simplex in the triangulation. The composition
of two PL homeomorphisms is again a PL homeomorphism by subdivisions of
triangulations. Let $\mathrm{Homeo}_{\mathrm{PL}}(M)$ be the set of all PL
homeomorphisms of $M.$ A group $G$ action on $M$ by PL homeomorphisms means
a group homomorphism $G\rightarrow \mathrm{Homeo}_{\mathrm{PL}}(M).$

Since there is no standard concepts of tangent spaces for PL manifolds, we
define a notion of tangent spheres (for our purpose) as an analogy for
tangent spaces, as follows.

\begin{definition}
\label{ts}Let $M^{n}$ be a manifold with a triangulation $\Sigma .$ Suppose
that $p$ is a vertex in $\Sigma .$ The tangent sphere $R_{\Sigma }$ at $p$
is the set of all rays with root $p$ in the star of $p.$
\end{definition}

The topology of tangent spheres is described in the following lemma. Note
that tangent spheres are not always spheres.

\begin{lemma}
\label{sph}Define a topology on the tangent sphere $R_{\Sigma }$ as the one
induced from polar coordinates. Then when $p\in M^{n}\backslash \partial
M^{n},$ $R_{\Sigma }$ is homeomorphic to $S^{n-1},$ the sphere of dimension $%
n-1;$ when $p\in \partial M,$ $R_{\Sigma }$ is homeomorphic to $D^{n-1},$ a
disk in $S^{n-1}$ which is bounded by a finite number of hyperplanes passing
through the origin$.$
\end{lemma}

\begin{proof}
This is obvious by noting that we could choose a unit vector in each ray.
\end{proof}

We consider group actions on tangent spheres. Let $G$ act on $M^{n}$ by PL
homeomorphisms. Suppose that $p$ is a global fixed point. For each element $%
g\in G,$ there is a triangulation $\Sigma _{g}$ (with $p$ as a vertex) on
which $g$ acts simplicially. Without loss of generality, we can fix a chart
of $p$ and view $p$ as a point in a Euclidean space. Since the action of $g$
is PL, the tangent sphere $R_{\Sigma _{g}}$is invariant under this action.
If $\Sigma _{g^{\prime }}$ is a subdivision of $\Sigma _{g},$ there is a
bijection $i_{g,g^{\prime }}:R_{\Sigma _{g}}\rightarrow R_{\Sigma
_{g^{\prime }}}$ by restrictions$.$ The map $i_{g,g^{\prime }}$ is a
homeomorphism.

Suppose that $B$ is a symmetric generating set ($g\in B$ implies $g^{-1}\in
B $) of a finitely generated group $G.$ We also assume that $B$ contains the
identity $1$. For each positive integer $i$, denote by $B^{i}$ the set of
products of $i$ elements in $B$ and $\Sigma _{B^{i}}$ the triangulation
generated by all $\Sigma _{g}$ for $g\in B^{i}.$ We see that $\Sigma
_{B^{i}} $ is a subdivision of $\Sigma _{B^{j}}$ for any $j<i.$ Similarly,
we can define $R_{\Sigma _{B^{i}}}.$ Denote by $\Sigma _{G}$ the direct
limit $\lim\limits_{i\rightarrow +\infty }\Sigma _{B^{i}}$ and $R_{p}$ the
direct limit of $\lim\limits_{i\rightarrow +\infty }R_{\Sigma _{B^{i}}}.$ We
call $R_{p}$ the tangent sphere of $M$ at $p$ with respect to the group
action of $G$ and a symmetric generating set $B.$

\begin{lemma}
Let a finitely generated group $G$ act on a manifold $M$ by PL
homeomorphisms with a global fixed point $p.$ There is an induced action of $%
G$ on the tangent sphere $R_{p}$ by PL homeomorphisms.
\end{lemma}

\begin{proof}
It is straightforward that the bijection $R_{\Sigma _{B^{i}}}\rightarrow
R_{\Sigma _{B^{i+1}}}$ is compactible with actions of $B^{i}.$ Therefore,
there is a group action of $G$ on $R_{p}$ induced from the action on $M.$
For each element $g\in G,$ $g\in B^{i}$ for some integer $i.$ By
construction, each element of $B^{i}$ acts on $R_{\Sigma _{B^{i}}}\cong
R_{p} $ piecewise linearly.
\end{proof}

\begin{remark}
It is not hard to see that the definition of $R_{p}$ is independent of
choices of generating sets$.$
\end{remark}

Note that $R_{p}$ is homeomorphic to $S^{n-1}$ when $p\notin \partial M$ or $%
D^{n-1}$ when $p\in \partial M.$

In order to prove Theorem \ref{rt}, we need the following two lemmas.

\begin{lemma}
\label{ray}Let $G$ be a finitely generated group and $H_{1}(G;\mathbb{R})=0.$
Then any action of $G$ on $[0,1)$ (resp. $[0,1]$) by PL homeomorphisms
fixing $0$ (resp. $\{0,1\}$) is trivial.
\end{lemma}

\begin{proof}
Since $0$ is a global fixed point, we could define a map $G\rightarrow 
\mathbb{R}^{\ast }$ by%
\begin{equation*}
g\mapsto (dg)(0),
\end{equation*}%
the one-sided derivative of $g$ at $0.$ This is a group homomorphism.
Therefore, a commutator $h$ in $G$ will be a constant function in a
neighborhood $[0,a_{h}]$ of $0.$ Since $H_{1}(G;\mathbb{R})=0$ and $G$ is
finitely generated, $H_{1}(G;\mathbb{Z})$ is finite and the commutator
subgroup $[G,G]$ is finitely generated. Suppose that elements $%
h_{1},h_{2},\cdots ,h_{k}$ generate $[G,G]$ and they together with other
elements $s_{1},s_{2},\cdots ,s_{m}$ generate the whole group $G.$ There is
a triangulation of $[0,1)$ on which each element of $\{h_{1},h_{2},\cdots
,h_{k},s_{1},s_{2},\cdots ,s_{m}\}$ acts piecewise linearly. Let $a$ be the
minimum positive vertex in the $0$-skeleton of this triangulation. Without
loss of generality, assume that $\{h_{1},h_{2},\cdots
,h_{k},s_{1},s_{2},\cdots ,s_{m}\}$ is symmetric and contains the identity.
Then $[G,G]$ acts trivially on $[0,a].$ Since $H_{1}(G;\mathbb{Z})$ is
finite and a group acting effectively on $[0,a]$ must be left orderable, the
action of $\{s_{1},s_{2},\cdots ,s_{m}\}$ is also trivial on $[0,a].$
Therefore, the action of $G$ on $[0,a]$ is trivial. We claim that there is
no largest such $a<1,$ which means the action on $[0,1)$ of $G$ is trivial.
On the contrary, suppose that $a_{0}$ is the largest such $a.$ Then the
group action of $G$ on $[a,1)$ (resp. $[a,1]$) is piecewise linear with $a$
a global fixed point. Repeat previous argument, we could find a new $%
a^{\prime }>a$ such that $G$ acts trivially on $[0,a^{\prime }].$ This is a
contradiction by the assumption of $a.$
\end{proof}

The following lemma is a high-dimensional analogy of the previous lemma.

\begin{lemma}
\label{cone}Let $G$ be a finitely generated group and $H_{1}(G;\mathbb{R})=0$
and $L$ a polytope. Suppose that $G$ acts on $L$ by PL homeomorphisms with
its boundary fixed. Then the action of $G$ on $R_{p}$ is trivial for any
vertex $p\in \partial L$.
\end{lemma}

\begin{proof}
The idea of the proof is similar to that of the previous one. Choose a
vertex $p$ of $L$ and consider the set of faces $\{F_{1},F_{2},\cdots
,F_{m}\}$ containing $p.$ For each such face $F_{i},$ we define a map $%
G\rightarrow \mathbb{R}^{n-1}\backslash \{0\}$ ($n$ is the dimension of $L$)
by%
\begin{equation*}
g\mapsto (dg)|_{F_{i}}(p),
\end{equation*}%
the directional derivative of $p$ along $F_{i}.$ Since $G$ acts trivially on 
$F_{i},$ this is a group homomorphism. Then every element in the commutator
subgroup will have image $(1,1,\ldots ,1).$ This implies that a commutator
acts trivially on a neighborhood of $p.$ Therefore, the action of $G$ on $%
R_{p}$ factors through $H_{1}(G;\mathbb{Z}),$ which is a finite abelian
group (since $H_{1}(G;\mathbb{R})=0$). However, any homeomorphism of a disk
with boundary fixed must be torsion-free (cf. \cite{sw}, p. 264). This
implies that $G$ acts trivially on $R_{p}.$
\end{proof}

\bigskip

In the following, we prove the PL version of Reeb-Thurston's stability.

\begin{proof}[Proof of Theorem \protect\ref{rt}]
Let $B$ be a finite symmetric generating set containing the identity of $G.$
Since $G$ acts trivially on the tangent sphere $R_{p},$ each ray in $%
R_{\Sigma _{B}}$ is invariant by the action of $B.$ Therefore, each element $%
g\in B$ maps a ray $r_{p}\in R_{\Sigma _{B}}$ to itself and so does each $%
g\in G$. Fix any ray $r_{p},$ which is homeomorphic to $[0,1).$ Then the
group $G$ acts piecewise linearly on $r_{p}.$ By Lemma \ref{ray}, $G$ acts
trivially on $r_{p}.$ Therefore, $G$ acts trivially on the star $\mathrm{St}%
_{p}$ of $p$ in the triangulation $\Sigma _{B}.$ Choose another vertex $%
p^{\prime }$ in the boundary of the star of $p$. By Lemma \ref{cone}, the
group action of $G$ on $(\overline{M\backslash \mathrm{St}_{p}})\cap
R_{p^{\prime }}$ is trivial and so is on $R_{p^{\prime }}.$ Repeating the
argument as above, this shows that the fixed point set $\mathrm{Fix}(G)$ is
both open and closed, which implies that $\mathrm{Fix}(G)$ is the whole
manifold $M.$
\end{proof}

\section{Actions on circles, surfaces and 3-manifolds}

We need a fixed-point result of Fuller \cite{Fu}:

\begin{lemma}
\label{fuller}Let $M$ be an orientable compact combinatorial manifold and $%
f:M\rightarrow M$ be a homeomorphism. If the Euler characteristic of $M$ is
not zero, then $f$ has a periodic point, i.e. $f^{k}(x)=x$ for some integer $%
k$ and some $x\in M.$
\end{lemma}

The following lemma gives a description of the fixed point set of a PL
homeomorphism.

\begin{lemma}
\label{char}Suppose that $M$ is a compact manifold. Let $f:M\rightarrow M$
be a PL homeomorphism with nonempty fixed point set $\mathrm{Fix}(f).$ Then

\begin{enumerate}
\item[(i)] $\mathrm{Fix}(f)$ is a finite simplicial complex;

\item[(ii)] there is a (not necessarily connected) closed manifold $%
N_{f}\subset \mathrm{Fix}(f)$ (of codimension at least $1$ if $f\neq \mathrm{%
id}_{M}$) such that for any PL homeomorphism $g:M\rightarrow M$ with $g(%
\mathrm{Fix}(f))=\mathrm{Fix}(f),$ we have $g(N_{f})=N_{f}.$
\end{enumerate}
\end{lemma}

\begin{proof}
Since $f$ is affine linear at each simplex $\sigma ,$ the intersection $%
\mathrm{Fix}(f)\cap \sigma $ is also a simplicial complex (as a solution of
a linear equation). Since $M$ is compact, $F=\mathrm{Fix}(f)$ is a union of
finite simplicial complexes and therefore a simplicial complex. If $f\neq 
\mathrm{id}_{M},$ $F\neq M.$ Since $F$ is closed and compact, the boundary $%
\partial \mathrm{Fix}(f)=F\backslash \mathring{F}$ is a manifold of
codimension at least $1,$ which is invariant under the action of $g.$ If $%
\partial \mathrm{Fix}(f)$ is a closed manifold, take $N_{f}=\partial \mathrm{%
Fix}(f).$ Otherwise, take $N_{f}=\partial (\partial \mathrm{Fix}(f)).$
\end{proof}

Witte \cite{Wi} proves that a lattice $\Gamma $ of $\mathbb{Q}$-rank at
least $2$ acts trivially on the circle $S^{1}$ by homeomorphisms. Ghys \cite%
{gh} and Burger-Monod \cite{bm} show that any lattice acts trivially on $%
S^{1}$ by $C^{1}$ homeomorphisms. The following result considers PL actions
for $1$-big lattices.

\begin{theorem}
\label{circle}Let $\Gamma $ be a $1$-big lattice. Then any action of $\Gamma 
$ on $S^{1}$ by PL homeomorphisms factors through a finite group action.
\end{theorem}

\begin{proof}
Let $\phi :\Gamma \rightarrow \mathrm{Home}_{\mathrm{PL}}(S^{1})$ be a group
homomorphism. If the $\mathbb{Q}$-rank of $\Gamma $ is at least $2,$ Witte 
\cite{Wi} shows that $\Gamma $ acts trivially on $S^{1}.$ Hence, by the
definition of $1$-big, we may assume that there is an element $\gamma \in
\Gamma $ such that the centralizer $C_{\Gamma }(\gamma )$ contains a
subgroup $\Lambda $ isomorphic to an irreducible lattice in some semisimple
Lie group of $\mathbb{R}$-rank at least $2.$ By passing to a subgroup of
index $2$ if necessary, we may assume that the action is
orientation-preserving.

Suppose that $\tilde{\gamma}:\mathbb{R}\rightarrow \mathbb{R}$ is a lifting
of $\phi (\gamma ).$ Define the \emph{rotation number} of $\phi (\gamma )$
as 
\begin{equation*}
\alpha =\lim\limits_{n\rightarrow +\infty }\gamma ^{n}(x)/n.
\end{equation*}%
Such number is independent of $x$ (cf. Theorem 1, p. 74 in \cite{cfs}).
Moreover, $\alpha $ is rational if and only if $\phi (\gamma )^{k}$ has a
fixed point for some integer $k.$ We prove the theorem in two cases.

\begin{enumerate}
\item[(1)] \noindent $\alpha $ is rational.
\end{enumerate}

This means that the fixed point set $\mathrm{Fix}(\phi (\gamma ^{k}))$ is
not empty. If $\gamma ^{k}\in \ker \phi ,$ the normal subgroup generated by $%
\gamma ^{k}$ in $\Gamma $ is of finite index. Therefore, $\phi $ factors
through a finite group. If $\phi (\gamma ^{k})\neq \mathrm{id}_{S^{1}},$
Lemma \ref{char} implies that there is a manifold $N$ of codimension at
least $1$ in $\mathrm{Fix}(\phi (\gamma ^{k}))$ and $N$ is invariant under
the action of $\Lambda .$ Since $S^{1}$ is compact, $N$ has only finitely
many components, i.e. $N$ is a set of finitely many points. Then for some
finite-index subgroup $\Lambda _{0}$ of $\Lambda ,$ $\Lambda _{0}$ has a
global fixed point. Since $\Lambda _{0}/[\Lambda _{0},\Lambda _{0}]$ is
finite and the group action is orientation preserving, the action $\Lambda
_{0}$ is trivial by the PL version of Reeb-Thurston's stability (cf. Theorem %
\ref{rt}). Note that the normal subgroup generated by $\Lambda _{0}$ in $%
\Gamma $ is of finite index. Therefore, $\phi $ factors through a finite
group.

\begin{enumerate}
\item[(2)] \noindent $\alpha $ is irrational.
\end{enumerate}

The proof is similar to that of Theorem 1.1 in Farb and Shalen \cite{fs}. We
briefly repeat it here for completeness. The Denjoy Theorem (cf. \cite{ms},
p.38) says that there is an element $\rho \in \mathrm{Home}(S^{1})$ such
that $T_{\alpha }:=\rho \phi (\gamma )\rho ^{-1}$ is the rotation by angle $%
\alpha .$ The cyclic group $\langle T_{\alpha }\rangle $ generated by $%
T_{\alpha }$ is dense in the group $R$ of rotations of $S^{1},$ since $%
\alpha $ is irrational. Then the subgroup $\rho \phi (\Lambda )\rho ^{-1}$
commutes with $\langle T_{\alpha }\rangle $ and thus commutes with $R.$
Since a homeomorphism commuting with any rotation must be a rotation, $\rho
\phi (\Lambda )\rho ^{-1}$ is a subgroup of $R.$ Therefore, there is a
noncentral element $\gamma ^{\prime }\in \Lambda $ satisfying $\phi (\gamma
^{\prime })=\mathrm{id}_{S^{1}}.$ The normal subgroup generating by $\gamma
^{\prime }$ is of finite index and therefore $\phi $ factors through a
finite group.
\end{proof}

\bigskip

\begin{proof}[Proof of Theorem \protect\ref{main}]
Suppose that $\phi :\Gamma \rightarrow \mathrm{Home}_{\mathrm{PL}}(M)$ is a
group homomorphism. First, let us consider the case when $M$ is a closed
surface other than the torus or Klein bottle. We will show that the image $%
\func{Im}\phi $ is finite. Suppose that there is an element $\gamma \in
\Gamma $ such that the centralizer $C_{\Gamma }(\gamma )$ contains a
subgroup $\Lambda $ isomorphic to a $1$-big irreducible lattice in some
semisimple Lie group$.$ Denote by $f=\phi (\gamma ).$ Since the Euler
characteristic of $M$ is nonzero, Lemma \ref{fuller} implies that for some
integer $k,$ the fixed point set $\mathrm{Fix}(f^{k})$ is not empty. Then
the subgroup $\Lambda $ acts on $\mathrm{Fix}(f^{k}).$ There two cases to
consider.

\begin{enumerate}
\item[(1)] \noindent $\mathrm{Fix}(f^{k})=M$.
\end{enumerate}

The element $\gamma ^{k}$ acts trivially on $M,$ i.e. $\gamma ^{k}\in \ker
\phi .$ The normal subgroup generated by $\gamma ^{k}$ is of finite index by
Magulis finiteness theorem. This shows that $\func{Im}\phi $ is finite.

\begin{enumerate}
\item[(2)] $\mathrm{Fix}(f^{k})\neq M.$
\end{enumerate}

Lemma \ref{char} implies that there is a manifold $N$ of codimension at
least $1$ in $\mathrm{Fix}(\phi (\gamma ^{k}))$ and $N$ is invariant under
the action of $\Lambda .$ If $\dim N=1,$ $N$ consists of finitely many
copies of $S^{1}.$ Therefore, some subgroup $\Lambda _{0}$ of finite index
in $\Lambda $ will act on a $S^{1}.$ By Theorem \ref{circle}, there exists a
finite-index normal subgroup $\Lambda ^{\prime }$ of $\Lambda _{0}$ acting
trivially on $S^{1}\subset N.$ Since the $\Lambda ^{\prime }$ is $1$-big, $%
\Lambda ^{\prime }$ acts trivially on the tangent sphere $S^{1}$ of a point $%
p\in S^{1}\subset N$ by Theorem \ref{circle} once again. Note that $%
H_{1}(\Lambda ^{\prime };\mathbb{R})=0.$ The PL version of Reeb-Thurston's
stability (cf. Theorem \ref{rt}) implies that the subgroup $\Lambda ^{\prime
}$ acts trivially on $M$. The normal subgroup generated by $\Lambda ^{\prime
}$ in $\Gamma $ is of finite index. Therefore $\func{Im}\phi $ is finite.
The case of $\dim N=0$ could be considered similarly.

If the boundary $\partial M\neq \varnothing ,$ $\partial M$ is a disjoint
union of finitely many circles (since $M$ is compact). It is clear that $%
\Gamma $ maps $\partial M$ to itself. Therefore for some finite-index
subgroup $\Gamma _{0}$ of $\Gamma ,$ a component $S^{1}$ of $\partial M$ is
invariant under the action of $\Gamma _{0}.$ Since $\Gamma _{0}$ is also a $%
2 $-big lattice, there exists an infinite-order element $\gamma ^{\prime
}\in \Gamma _{0}$ such that the centralizer $C_{\Gamma _{0}}(\gamma ^{\prime
})$ contains a subgroup $G,$ which is $1$-big. Theorem \ref{circle} implies
that there is a finite-index subgroup $G_{0}$ of $G$ acting trivially on $%
S^{1}.$ For a point $p\in S^{1}\subset \partial M,$ the tangent sphere $%
R_{p} $ is homeomorphic to the interval $[0,1]$ (cf. Lemma \ref{sph}). By
Lemma \ref{ray}, the group $G_{0}$ acts trivially on $R_{p}.$ Similar to
case $(2), $ Theorem \ref{rt} implies that $G_{0}$ acts trivially on $M.$
The normal subgroup in $\Gamma $ generated by $G_{0}$ is of finite index.
Therefore, $\phi (\Gamma )$ is finite. The proof is finished.
\end{proof}

\bigskip

\begin{proof}[Proof of Theorem \protect\ref{three}]
Since $M$ is compact and the Euler characteristic $\chi (\partial M)\neq 0,$
there is a connected closed surface $\Sigma \subset \partial M$ with $\chi
(\Sigma )\neq 0.$ It is obvious that the boundary $\partial M$ is invariant
under the action of $\Gamma .$ Therefore for some finite-index subgroup $%
\Gamma _{0},$ $\Sigma $ is invariant under the action of $\Gamma _{0}.$
Since $\Gamma _{0}$ is $2$-big, Theorem \ref{main} implies that the action
is trivial for some finite-index subgroup $\Lambda $ of $\Gamma _{0}.$ It's
not hard to see that the tangent sphere $R_{p}$ for $p\in \Sigma $ is
homeomorphic to a disk $D^{2}$ (cf. Lemma \ref{sph}), whose boundary is
fixed by $\Lambda .$ By Theorem \ref{rt} and Lemma \ref{ray}, $\Lambda $
acts trivially on $\Sigma .$ Using Theorem \ref{rt} once again, $\Lambda $
acts trivially on $M.$ The normal subgroup generated by $\Lambda $ in $%
\Gamma $ is of finite index$.$ Therefore, $\Gamma $ factors through a finite
group.
\end{proof}

\bigskip

\noindent \textbf{Acknowledgements}

The author is grateful to Marc Lackenby, Will Cavendish at Oxford and Feng
Ji, Zihong Yuan at Singapore for many helpful discussions.

\bigskip

Mathematical Institute, University of Oxford, 24-29 St Giles', Oxford, OX1 3LB, UK. 

E-mail: Shengkui.Ye@maths.ox.ac.uk

\end{document}